\documentclass[conference]{IEEEtran}

\usepackage{cite}
\usepackage{amsmath,amssymb,amsfonts}
\usepackage{algorithmic}
\usepackage{graphicx}
\graphicspath{{./figures/}}
\usepackage{textcomp}
\usepackage{xcolor}

\def\bsg{{\boldsymbol{g}}}

\def\bsv{{\boldsymbol{v}}}

\def\bsx{{\boldsymbol{x}}}

\def\bsH{{\boldsymbol{H}}}

\def\bsK{{\boldsymbol{K}}}
\def\bsL{{\boldsymbol{L}}}

\def\bsM{{\boldsymbol{M}}}

\def\bsQ{{\boldsymbol{Q}}}

\usepackage{bbm}
\usepackage{setspace}
\usepackage{array}
\usepackage{booktabs}
\usepackage{hyperref}
\usepackage{bm}
\usepackage{adjustbox}
\usepackage{makecell}

\usepackage[T1]{fontenc}
\usepackage{pifont}
\usepackage{psfrag,cite,amsmath,epsfig,latexsym,subfigure,color}

\usepackage{algorithm,algorithmic}

\newtheorem{theorem}{Theorem}
\newtheorem{assumption}{Assumption}

\newtheorem{lemma}{Lemma}

\newtheorem{definition}{Definition}
\newtheorem{remark}{Remark}

\usepackage{tikz}
\usetikzlibrary{decorations.pathreplacing}
\usetikzlibrary{fadings}

\def\sgn{\mathop{\rm sgn}}

\newcommand{\diag}{{\rm diag}}

\DeclareMathOperator{\col}{col}

\DeclareMathOperator{\Deg}{Deg}

\usepackage{array}
\newcolumntype{M}[1]{>{\centering\arraybackslash}m{#1}}
\newcolumntype{N}{@{}m{0pt}@{}}

\usepackage{multirow}

\usepackage{color}
\definecolor{gray}{RGB}{128,128,128}

\newcommand{\red}[1]{{\color{red} #1}}

\usepackage{cite}

\newenvironment{proof}[1][Proof]%
  {\smallskip\par\noindent\textbf{#1\,:\ }}%
  {\hspace*{\fill} \rule{6pt}{6pt}\smallskip}
\newenvironment{proof*}[1][Proof]%
  {\smallskip\par\noindent\textbf{#1\,:\ }}%

\def\BibTeX{{\rm B\kern-.05em{\sc i\kern-.025em b}\kern-.08em
    T\kern-.1667em\lower.7ex\hbox{E}\kern-.125emX}}
    
\allowdisplaybreaks
\begin{document}

\title{Accelerated Primal-Dual Algorithm for Distributed Nonconvex Optimization
}

\author{\IEEEauthorblockN{Shengjun~Zhang}
\IEEEauthorblockA{\textit{OSCAR, Department of Electrical Engineering} \\
\textit{University of North Texas}\\
Denton, TX, USA\\
ShengjunZhang@my.unt.edu}
\and
\IEEEauthorblockN{Colleen P. Bailey}
\IEEEauthorblockA{\textit{OSCAR, Department of Electrical Engineering} \\
\textit{University of North Texas}\\
Denton, TX, USA \\
Colleen.Bailey@unt.edu}
}

\maketitle

\begin{abstract}
This paper investigates accelerating the convergence of distributed optimization algorithms on non-convex problems.
We propose a distributed primal-dual stochastic gradient descent~(SGD) equipped with ``powerball'' method to accelerate.
We show that the proposed algorithm achieves the linear speedup convergence rate $\mathcal{O}(1/\sqrt{nT})$ for general smooth (possibly non-convex) cost functions.
We demonstrate the efficiency of the algorithm through numerical experiments by training two-layer fully connected neural networks and convolutional neural networks on the MNIST dataset to compare with state-of-the-art distributed SGD algorithms and centralized SGD algorithms.
\end{abstract}


\section{Introduction}\label{fo_pb:intro}

Centralized stochastic gradient descent~(SGD) is one of the most popular and powerful optimizers.
With the rise of distributed optimization problems, centralized SGD has been extended to distributed settings.
A general distributed optimization problem is to consider a network of $n$ agents cooperatively solving a global problem, which can be formulated mathematically as~\eqref{fo_pb:eqn:xopt}.
\begin{align}\label{fo_pb:eqn:xopt}
 \min_{x\in \mathbb{R}^p} f(x)=\frac{1}{n}\sum_{i=1}^nf_i(x),
\end{align}
where $x\in \mathbb{R}^p$ is the optimization variables, $f_i: \mathbb{R}^{p}\rightarrow \mathbb{R}$ is the local smooth (possibly non-convex) cost function of agent $i$.
Each agent can only exchange information~(e.g. local gradient information) with its own neighbors through a given undirected graph $\mathcal G$.

Many distributed machine learning problems can be formulated into the optimization problem in the form of~\eqref{fo_pb:eqn:xopt} through data parallelism, such as deep learning \cite{dean2012large} and federated learning \cite{McMahan2017communication}.
In this paper, we consider the general case without focusing on a particular application and propose an accelerated primal-dual distributed algorithm to solve~\eqref{fo_pb:eqn:xopt}.

\subsection{Literature Review}
Many distributed algorithms based on SGD have been proposed in recent years.
Various parallel SGD algorithms aim to solve~\eqref{fo_pb:eqn:xopt} when the communication network is a star graph;
parallel asynchronous updating SGD algorithms \cite{recht2011hogwild,de2015taming,lian2015asynchronous,lian2016Comprehensive,Zhou2018distributedas},
parallel SGD algorithms with compression\cite{de2015taming,pmlr-v80-bernstein18a,
jiang2018linear,reisizadeh2019fedpaq,basu2019qsparse},
parallel SGD algorithms with periodic averaging approaches \cite{jiang2018linear,reisizadeh2019fedpaq,wang2018adaptive,yu2019parallel,
haddadpour2019trading,Yu2019on,haddadpour2019local}, and parallel SGD algorithm with adaptive batch sizes \cite{Yu2019Computation}.
In \cite{jiang2018linear,yu2019parallel,Yu2019on,Yu2019Computation}, the authors showed a linear speedup convergence rate of $\mathcal{O}(1/\sqrt{nT})$ for general non-convex cost functions, where $T$ is the total number of iterations.
However, the aforementioned parallel SGD algorithms require additional assumptions such as bounded gradients of global cost functions.

Compared to parallel SGD algorithms, distributed SGD algorithms naturally overcome communication bottlenecks.
In this category, a great number of algorithms have been proposed in order to solve~\eqref{fo_pb:eqn:xopt} efficiently and accurately;
\cite{Yu2019on,jiang2017collaborative,lian2017can,george2019distributed} proposed synchronous distributed SGD algorithms,
\cite{pmlr-v80-lian18a,Assran2019Stochastic} considered asynchronous distributed SGD algorithms, 
\cite{tang2018communication,reisizadeh2019robust,taheri2020quantized,Singh2020Communication}
utilized compression techniques on distributed SGD algorithms, and\cite{wang2018cooperative} applied the periodic averaging method on distributed SGD algorithm.
The convergence rate has been established as $\mathcal{O}(1/\sqrt{nT})$ for general non-convex cost functions \cite{Yu2019on,lian2017can,Assran2019Stochastic,tang2018communication,taheri2020quantized,
Singh2020Communication,wang2018cooperative} with additional assumptions on the cost functions.
Based on \cite{lian2017can}, the authors of \cite{Tang2018Decentralized} proposed named $\mathrm{D}^2$ with convergence rate $\mathcal{O}(1/\sqrt{nT})$ without additional assumptions on the cost functions but under a restrictive assumption on the communication graph. 
Several distributed stochastic gradient tracking algorithms were proposed in \cite{lu2019gnsd,zhang2019decentralized} for arbitrarily connected communication networks with $\mathcal{O}(1/\sqrt{T})$ convergence rate.

There are only a few accelerated first-order distributed algorithms in the literature.
\cite{xu2020accelerated} proposed an accelerated primal-dual algorithms for distributed smooth convex optimization utilizing Nesterov acceleration.
\cite{li2020decentralized} considered solving distributed convex optimization problems with increasing penalty parameters to accelerate.
The authors of \cite{qu2019accelerated} considered Nesterov acceleration on different classes of convex cost functions.
\cite{uribe2020dual} proposed an accelerated dual algorithm for general convex cost functions.
These algorithms utilize full gradient information. In the era of big data, full gradient information is often too difficult to compute.
In \cite{fallah2019robust, Yu2019on}, the authors consider accelerating distributed SGD algorithm with momentum terms.
DM-SGD in~\cite{Yu2019on} has demonstrated a convergence rate of $\mathcal{O}(1/\sqrt{nT})$ for general non-convex cost functions but with similar local cost functions.
D-ASG in~\cite{fallah2019robust} achieved a convergence rate of $\mathcal{O}(1/\sqrt{nT})$ for convex cost functions.
\subsection{Contributions}
The contributions of this work are summarized in the following:
\subsubsection{}
We propose an accelerated distributed primal-dual SGD algorithm to solve the optimization problem~\eqref{fo_pb:eqn:xopt} for arbitrarily connected communication networks.
Even though each agent computes its own primal and dual variable in each iteration, only the primal variable is shared with its neighbors.
\subsubsection{}
To our best knowledge, the proposed algorithm is the first accelerated distributed SGD for general non-convex cost functions, which is different from \cite{fallah2019robust}.
As opposed to \cite{Yu2019on}, there is no additional assumptions on cost functions.
\subsubsection{}
Theoretically, we show that the proposed algorithm achieves a linear speedup convergence rate $\mathcal{O}(1/\sqrt{nT})$.
Moreover, we compare the proposed algorithm with several state-of-the-art algorithms in distributed machine learning and deep learning tasks to illustrate the advantages of utilizing the proposed algorithm.

\subsection{Outline}
The rest of this paper is organized as follows.
Section~\ref{fo_pb:pre} introduces some preliminary concepts. Sections~\ref{fo_pb:alg} introduces the proposed algorithm and analyzes its convergence properties. Simulations are presented in Section~\ref{fo_pb:exp}. Finally, concluding remarks are offered in Section~\ref{fo_pb:con}. 

\noindent {\bf Notations}: $\mathbb{N}_0$ and $\mathbb{N}_+$ denote the set of nonnegative and positive integers, respectively. $[n]$ denotes the set $\{1,\dots,n\}$ for any $n\in\mathbb{N}_+$.
$\|\cdot\|$ represents the Euclidean norm for vectors or the induced 2-norm for matrices. 
$\|\cdot\|_{p}$ denotes the $p$-norm for vectors.
Given a differentiable function $f$, $\nabla f$ denotes the gradient of $f$.
${\bm 1}_n$ (${\bm 0}_n$) denotes the column one (zero) vector of dimension $n$. $\col(z_1,\dots,z_k)$ is the concatenated column vector of vectors $z_i\in\mathbb{R}^{p_i},~i\in[k]$. ${\bm I}_n$ is the $n$-dimensional identity matrix. Given a vector $[x_1,\dots,x_n]^\top\in\mathbb{R}^n$, $\diag([x_1,\dots,x_n])$ is a diagonal matrix with the $i$-th diagonal element being $x_i$.  The notation $A\otimes B$ denotes the Kronecker product
of matrices $A$ and $B$. Moreover, we denote $\bsx=\col(x_1,\dots,x_n)$, $\bar{x}=\frac{1}{n}({\bm 1}_n^\top\otimes{\bm I}_p)\bsx$, $\bar{\bsx}={\bm 1}_n\otimes\bar{x}$.
$\rho(\cdot)$ stands for the spectral radius for matrices and $\rho_2(\cdot)$ indicates the minimum
positive eigenvalue for matrices having positive eigenvalues.

\section{Preliminaries}\label{fo_pb:pre}
\subsection{Graph Theory}
Agents communicate with their neighbors through an underlying network, which is modeled by an undirected graph $\mathcal G=(\mathcal V,\mathcal E)$, where $\mathcal V =\{1,\dots,n\}$ is the agent set, $\mathcal E
\subseteq \mathcal V \times \mathcal V$ is the edge set, and $(i,j)\in \mathcal E$ if agents $i$ and $j$ can communicate with each other.
For an undirected graph $\mathcal G=(\mathcal V,\mathcal E)$, let $\mathcal{A}=(a_{ij})$ be the associated weighted adjacency matrix with $a_{ij}>0$ if $(i,j)\in \mathcal E$ if $a_{ij}>0$ and zero otherwise. It is assumed that $a_{ii}=0$ for all $i\in [n]$. Let $\deg_i=\sum\limits_{j=1}^{n}a_{ij}$ denotes the weighted degree of vertex $i$. The degree matrix of graph $\mathcal G$ is $\Deg=\diag([\deg_1, \cdots, \deg_n])$. The Laplacian matrix is $L=(L_{ij})=\Deg-\mathcal{A}$. Additionally, we denote $K_n={\bm I}_n-\frac{1}{n}{\bm 1}_n{\bm 1}^{\top}_n$, $\bsL=L\otimes {\bm I}_p$, $\bsK=K_n\otimes {\bm I}_p$, $\bsH=\frac{1}{n}({\bm 1}_n{\bm 1}_n^\top\otimes{\bm I}_p)$. Moreover, from Lemmas~1 and 2 in \cite{Yi2018distributed}, we know there exists an orthogonal matrix $[r \ R]\in \mathbb{R}^{n \times n}$ with $r=\frac{1}{\sqrt{n}}\mathbf{1}_n$ and $R \in \mathbb{R}^{n\times (n-1)}$ such that $R\Lambda_1^{-1}R^{\top}L=LR\Lambda_1^{-1}R^{\top}=K_n$, and $\frac{1}{\rho(L)}K_n\leq R\Lambda_1^{-1}R^{\top}\le\frac{1}{\rho_2(L)}K_n$, where $\Lambda_1=\diag([\lambda_2,\dots,\lambda_n])$ with $0<\lambda_2\leq\dots\leq\lambda_n$ being the eigenvalues of the Laplacian matrix $L$.

\subsection{Smooth Function}
\begin{definition}
A function $f(x):~\mathbb{R}^p\mapsto\mathbb{R}$ is smooth with constant $L_f>0$ if it is differentiable and
\begin{align}\label{nonconvex:smooth}
\|\nabla f(x)-\nabla f(y)\|\le L_{f}\|x-y\|,~\forall x,y\in \mathbb{R}^p.
\end{align}
\end{definition}

\subsection{Powerball Term}\label{fo_pb:pbterm}
Define the function 
\begin{equation}\label{pb:def}
\sigma(x, \gamma) = \sgn(x) |x|^{\gamma}
\end{equation}
where $\gamma \in [\frac{1}{2}, 1]$.

Note that when $\gamma = 1$, $\sigma(x, 1)$ reduces to $x$.
Unlike the ``powerball'' terms in \cite{zhou2020pbsgd} and \cite{yuan2019powerball}, under distributed settings, the range of $\gamma$ has to be modified.
\subsection{Assumptions}

\begin{assumption}\label{fo_pb:ass:graph}
The undirected graph $\mathcal G$ is connected.
\end{assumption}

\begin{assumption}\label{fo_pb:ass:optset}
The optimal set $\mathbb{X}^*$ is nonempty and the optimal value $f^*>-\infty$.
\end{assumption}

\begin{assumption}\label{fo_pb:ass:local}
Each local cost function $f_i$ is smooth with constant $L_{f}>0$, i.e.,
\begin{align}
\|\nabla f_i(x)-\nabla f_i(y)\|\le L_{f}\|x-y\|,~\forall x,y\in \mathbb{R}^p.
\end{align}
\end{assumption}

\begin{assumption}\label{fo_pb:ass:stochastic-grad:xi}
The random variables $\xi_{i,k}$ in agent $i$, where $i\in[n],~k\in\mathbb{N}_0$ are independent of each other.
\end{assumption}

\begin{assumption}\label{fo_pb:ass:stochastic-grad:mean}
The stochastic estimate $\nabla_{i}(x,\xi_{i,k})$ is unbiased, i.e., for all $i\in[n]$, $k\in\mathbb{N}_0$, and $x\in\mathbb{R}^p$,
\begin{align}\label{sgd:ass:stochastic-grad:mean-equ}
\mathbb{E}_{\xi_{i,k}}[\nabla_{i}(x,\xi_{i,k})]=\nabla f_{i}(x).
\end{align}
\end{assumption}

\begin{assumption}\label{fo_pb:ass:stochastic-grad:variance}
The stochastic estimate $\nabla_{i}(x,\xi_{i,k})$ has bounded variance, i.e., there exists a constant $\sigma$ such that for all $i\in[n]$, $k\in\mathbb{N}_0$, and $x\in\mathbb{R}^p$,
\begin{align}\label{fo_pb:ass:stochastic-grad:variance-equ}
\mathbb{E}_{\xi_{i,k}}[\|\nabla_{i}(x,\xi_{i,k})-\nabla f_{i}(x)\|^2]\le\sigma^2.
\end{align}
\end{assumption}

\section{Proposed Algorithm}\label{fo_pb:alg}

\subsection{Algorithm Description}
We consider the novel distributed primal-dual scheme~\cite{yi2019linear, yi2020primal} with a mini-batch variation in~\eqref{fopb:sgd-minibatch} to get the estimations of gradient per agent per iteration.
\begin{equation}\label{fopb:sgd-minibatch}
\nabla_{i}(x_{i,k},\xi_{i,k}) = \frac{1}{B}\sum_{b = 1}^{B} \nabla_{i}(x_{i,k, b},\xi_{i,k, b})
\end{equation}
Then we apply the ``powerball'' term described in~\eqref{pb:def} directly on the  estimations of gradient.
We summarize the proposed method as Algorithm~\ref{fopb:algorithm-random-pd}.

\begin{algorithm}[!ht]
\caption{DSGPA-F-PB}
\label{fopb:algorithm-random-pd}
\begin{algorithmic}[1]
\STATE \textbf{Input}: positive number $\alpha$, $\beta$, $\eta$, and $\gamma$.
\STATE \textbf{Initialize}: $ x_{i,0}\in\mathbb{R}^p$ and $v_{i,0}={\bm 0}_p,~
\forall i\in[n]$.
\FOR{$k=0,1,\dots$}
\FOR{$i=1,\dots,n$  in parallel}
\STATE  Broadcast $x_{i,k}$ to $\mathcal{N}_i$ and receive $x_{j,k}$ from $j\in\mathcal{N}_i$;
\STATE  Sample stochastic gradient $\nabla_{i}(x_{i,k},\xi_{i,k})$;
\STATE  Update $x_{i,k+1}$ by \eqref{fopb:alg:random-pd-x};
\STATE  Update $v_{i,k+1}$ by \eqref{fopb:alg:random-pd-q}.
\ENDFOR
\ENDFOR
\STATE  \textbf{Output}: $\{\bsx_{k}\}$.
\end{algorithmic}
\end{algorithm}

\begin{subequations}\label{fopb:alg:random-pd}
\begin{align}
x_{i,k+1} &= x_{i,k}-\eta\Big(\alpha\sum_{j=1}^nL_{ij}x_{j,k}+\beta v_{i,k}\nonumber \\+&\sigma\Big(\nabla_{i}(x_{i,k},\xi_{i,k}), \gamma\Big)\Big),  \label{fopb:alg:random-pd-x}\\
v_{i,k+1} &=v_{i,k}+ \eta\beta\sum_{j=1}^n L_{ij}x_{j,k}, \nonumber\\ &\forall x_{i,0}\in\mathbb{R}^p, ~\sum_{j=1}^nv_{j,0}={\bm 0}_p,~
\forall i\in[n].  \label{fopb:alg:random-pd-q}
\end{align}
\end{subequations}

\subsection{Convergence Analysis}

\begin{theorem}\label{fo_pb:thm}
Suppose Assumptions~\ref{fo_pb:ass:graph}--\ref{fo_pb:ass:stochastic-grad:variance} hold. 
For any given $T>n^3$, let $\{\bsx_k,k=0,\dots,T\}$ be the output sequence generated by Algorithm~\ref{fopb:algorithm-random-pd} with
\begin{align}\label{fo_pb:para}
\alpha=\kappa_1\beta,~\beta=\kappa_2\sqrt{T}/\sqrt{n},~ \eta=\frac{\kappa_2}{\beta},~\forall k\in\mathbb{N}_0,
\end{align}
where $\kappa_1>\frac{1}{\rho_2(L)}+1$, and $\kappa_2\in\Big(0,\min\{\frac{(\kappa_1-1)\rho_2(L)-1}{\rho(L)+(2\kappa_1^2+1)\rho(L^2)+1},\frac{1}{5}\}\Big)$, then we have,
\begin{subequations}
\begin{align}
&\frac{1}{T}\sum_{k=0}^{T-1}\mathbb{E}\Big[\frac{1}{n}\sum_{i=1}^{n}\|x_{i,k}-\bar{x}_k\|^2\Big]
=\mathcal{O}(\frac{n}{T}),\label{fo_pb:coro-sg-sm-equ3.1}\\
&\frac{1}{T}\sum_{k=0}^{T-1}\mathbb{E}[\|\nabla f(\bar{x}_k)\|_{1+\gamma}^2]
={\mathcal{O}(\frac{1}{\sqrt{nT}})+\mathcal{O}(\frac{n}{T})},\label{fo_pb:coro-sg-sm-equ3}\\
&\mathbb{E}[f(\bar{x}_{T})]-f^*=\mathcal{O}(1).\label{fo_pb:coro-sg-sm-equ4}
\end{align}
\end{subequations}
where $\bar{x}_k=\frac{1}{n}\sum_{i=1}^{n}x_{i,k}$.
\end{theorem}

Before proving Theorem~\ref{fo_pb:thm}, we introduce the following useful lemmas.

\begin{lemma}\label{fo_pb:lemma:pb}
By using the powerball term in~\eqref{pb:def} and when $\gamma \in [\frac{1}{2}, 1)$, we have $\Big\Vert\sigma\Big(\nabla_{i}(x_{i,k},\xi_{i,k}), \gamma\Big)\Big\Vert^2 \leq \Big\Vert \nabla_{i}(x_{i,k},\xi_{i,k}) \Big\Vert^{2}_{1+\gamma}$.
\end{lemma}
\begin{proof}
Consider $\gamma \in [\frac{1}{2}, 1)$, let $c = \frac{1+\gamma}{1-\gamma}$, $d = \frac{1+\gamma}{2\gamma}$ and apply H\"{o}lder's inequality, we have
\begin{align*}
&\Big\Vert\sigma\Big(\nabla_{i}(x_{i,k},\xi_{i,k}), \gamma\Big)\Big\Vert^2 = \sum ^{p}_{l = 1}\Big|\Big(\nabla_{i}(x_{i,k},\xi_{i,k})\Big)_{l}\Big|^{2\gamma} \\
&\leq (\sum ^{p}_{l = 1} 1^{c})^{\frac{1}{c}} \Big(\sum ^{p}_{l = 1}\Big(\Big|\Big(\nabla_{i}(x_{i,k},\xi_{i,k})\Big)_{l}\Big|^{2\gamma}\Big)^{d}\Big)^{\frac{1}{d}} \\
&\leq \Vert \bm{1}\Vert_{c} \Big(\sum ^{p}_{l = 1} \Big|\Big(\nabla_{i}(x_{i,k},\xi_{i,k})\Big)_{l}\Big|^{1+\gamma}\Big)^{\frac{2\gamma}{1 + \gamma}} \\
& = \Vert \bm{1}\Vert_{c} \Big\Vert \nabla_{i}(x_{i,k},\xi_{i,k}) \Big\Vert^{2\gamma}_{1+\gamma}\\
& \leq \Big\Vert \nabla_{i}(x_{i,k},\xi_{i,k}) \Big\Vert^{2}_{1+\gamma}
\end{align*}
where $\Big(\nabla_{i}(x_{i,k},\xi_{i,k})\Big)_{l}$ is the $l$-th coordinate of $\nabla_{i}(x_{i,k},\xi_{i,k})$.
\end{proof}

\begin{lemma}\label{fo_pb:lemma:lya}
Suppose Assumptions~\ref{fo_pb:ass:graph} -~\ref{fo_pb:ass:stochastic-grad:variance} hold, and we have parameters chosen as in~\eqref{fo_pb:para}, for simplicity, we denote $g_{i, k}^{u} = \nabla_{i}(x_{i,k},\xi_{i,k})$ for agent $i$ at iteration $k$, and $\bsg^u_k=\col(g^u_{1,k},\dots,g^u_{n,k})$, 
$\bsg^0_k=n \nabla f(\bar{\bsx}_{k})$, $\bar{\bsg}_k^0=\bsH\bsg^0_{k}$.
Let $\{\bsx_k\}$ be the sequence generated by Algorithm~\ref{fopb:algorithm-random-pd}, then we have
\begin{subequations}
\begin{align}
\mathbb{E}[W_{k+1}]
&\le   W_{k}-\kappa_4\|\bsx_k\|^2_{\bsK}\nonumber\\
&\quad-\frac{1}{2}(\kappa_2 - 5\kappa_2^2)\Big\|\bm{v}_k+\frac{1}{\beta}\bsg_{k}^0\Big\|^2_{\bsK}\nonumber\\
&\quad-\frac{1}{4}\eta\|\bar{\bsg}^0_{k}\|_{1+\gamma}^2+\mathcal{O}(n)\sigma^2\eta^2,\label{fo_pb:sgproof-vkLya2T}
\end{align}
\begin{align}
&\mathbb{E}[W_{4,k+1}] \le W_{4,k} + \|\bsx_k\|^2_{\frac{1}{2}\eta L_f^2\bsK} + L_f^{2}\sigma^{2}\eta^{2},\label{fo_pb:sgproof-vkLya2T_W4_og}
\end{align}
\begin{align}
&\mathbb{E}\Big[W_{k+1}+\frac{\kappa_4}{L_f^{2}}W_{4, k+1}\Big]
\le W_{k}-\frac{1}{2}\kappa_4\|\bsx_k\|^2_{\bsK}\nonumber\\
&\quad-\frac{1}{2}(\kappa_2 - 5\kappa_2^2)\Big\|\bm{v}_k+\frac{1}{\beta}\bsg_{k}^0\Big\|^2_{\bsK}\nonumber\\
&\quad-\frac{1}{4}\frac{\kappa_4}{L_f^{2}}\|\bar{\bsg}^0_{k}\|_{1+\gamma}^2+\mathcal{O}(n)\sigma^2\eta^2 + \frac{\kappa_4}{L_f}\sigma^{2}\eta,\label{fo_pb:sgproof-vkLya2T_W4}
\end{align}
\end{subequations}
\end{lemma}

\begin{proof}
The proof are given in Appendix~\ref{fo_pb:lyap_lemma}.
\end{proof}

We are now ready to prove Theorem~\ref{fo_pb:lemma:pb}.
\begin{proof}
Denote
\begin{align*}
\hat{V}_k=\|\bm{x}_k\|^2_{\bsK}+\Big\|\bsv_k
+\frac{1}{\beta_k}\bsg_k^0\Big\|^2_{\bsK}+n(f(\bar{x}_k)-f^*).
\end{align*}
We have
\begin{align}
&W_{k}\nonumber\\
&=\frac{1}{2}\|\bsx_{k}\|^2_{\bsK}
+\frac{1}{2}\Big\|\bsv_k+\frac{1}{\beta}\bsg_k^0\Big\|^2_{\bsQ+\kappa_1\bsK}\nonumber\\
&~~~+\bsx_k^\top\bsK\Big(\bm{v}_k+\frac{1}{\beta}\bsg_k^0\Big)+n(f(\bar{x}_k)-f^*)\nonumber\\
&\ge\frac{1}{2}\|\bsx_{k}\|^2_{\bsK}
+\frac{1}{2}\Big(\frac{1}{\rho(L)}+\kappa_1\Big)\Big\|\bsv_k+\frac{1}{\beta}\bsg_k^0\Big\|^2_{\bsK}\nonumber\\
&~~~-\frac{1}{2\kappa_1}\|\bsx_{k}\|^2_{\bsK}
-\frac{\kappa_1}{2}\Big\|\bsv_k+\frac{1}{\beta}\bsg_k^0\Big\|^2_{\bsK}
+n(f(\bar{x}_k)-f^*)\nonumber\\
&\ge\min\Big\{\frac{1}{2\rho(L)},~\frac{\kappa_1-1}{2\kappa_1}\Big\}\hat{V}_k\ge0,\label{fo_pb:vkLya3}
\end{align}
Additionally, we can get $W_{k}\le(\frac{\kappa_1+1}{2}+\frac{1}{2\rho_2(L)})\hat{V}_k$.

From~\eqref{fo_pb:sgproof-vkLya2T}, we have
\begin{align}\label{fo_pb:vkLya4}
&\mathbb{E}[W_{k+1}]\nonumber\\
&\le W_{k}-\kappa_4\|\bsx_k\|^2_{\bsK}
-\frac{\kappa_2}{4\beta}\|\bar{\bsg}^0_{k}\|_{1+\gamma}^2
+\frac{\mathcal{O}(n)\kappa_2^2\sigma^2}{\beta^2}.
\end{align}
Then, taking expectation and summing \eqref{fo_pb:vkLya4} over $ k\in[0,T]$  yield
\begin{align}\label{fo_pb:vkLya4.1}
&\mathbb{E}[W_{T+1}]+\sum_{k=0}^{T}\mathbb{E}\Big[\kappa_4\|\bsx_k\|^2_{\bsK}
+\frac{\kappa_2}{4\beta}\|\bar{\bsg}^0_{k}\|_{1+\gamma}^2\Big]\nonumber\\
&\le W_{0}+\frac{(T+1)\mathcal{O}(n)\kappa_2^2\sigma^2}{\beta^2}.
\end{align}

From~\eqref{fo_pb:vkLya3},~\eqref{fo_pb:vkLya4.1}, and $\kappa_4 > 0$, then we have
\begin{align}\label{fo_pb:thm-sg-sm-equ1.1p}
&\frac{1}{T+1}\sum_{k=0}^{T}\mathbb{E}\Big[\frac{1}{n}\sum_{i=1}^{n}\|x_{i,k}-\bar{x}_k\|^2\Big]\nonumber\\
&
=\frac{1}{n(T+1)}\sum_{k=0}^{T}\mathbb{E}[\|\bsx_k\|^2_{\bsK}]\nonumber\\
&\le\frac{W_{0}}{n\kappa_4(T+1)}
+\frac{\mathcal{O}(n)\kappa_2^2\sigma^2}{n\kappa_4\beta^2}.
\end{align}
Consider that $W_0=\mathcal{O}(n)$ and $\beta=\kappa_2\sqrt{T}/\sqrt{n}$, from \eqref{fo_pb:thm-sg-sm-equ1.1p}, we have~\eqref{fo_pb:coro-sg-sm-equ3.1}.

Taking expectation and summing~\eqref{fo_pb:sgproof-vkLya2T_W4_og} over $ k\in[0,T]$  yield
\begin{align}\label{fo_pb:thm-sg-sm-equ2p}
&n(\mathbb{E}[f(\bar{x}_{T+1})]-f^*)=\mathbb{E}[W_{4,T+1}]\nonumber\\
&\le W_{4,0}+\frac{1}{2}\eta L_f^2\sum_{k=0}^{T}\mathbb{E}[\|\bsx_k\|^2_{\bsK}]
+L_f\sigma^2\eta^2(T+1).
\end{align}
From \eqref{fo_pb:vkLya4.1}, \eqref{fo_pb:thm-sg-sm-equ2p}, $\eta=\kappa_2/\beta$, and $\beta=\kappa_2\sqrt{T}/\sqrt{n}$, then we have \eqref{fo_pb:coro-sg-sm-equ4}.

Similar to get~\eqref{fo_pb:coro-sg-sm-equ3.1}, taking expectation and summing \eqref{fo_pb:sgproof-vkLya2T_W4} over $ k\in[0,T]$  yield
\begin{align}\label{fo_pb:vkLya4.1-speed}
&\frac{1}{4}n\frac{\kappa_4}{L_f^{2}}\sum_{k=0}^{T}\mathbb{E}[\|\nabla f(\bar{x}_k)\|_{1+\gamma}^2]=\frac{1}{4}\frac{\kappa_4}{L_f^{2}}\sum_{k=0}^{T}\mathbb{E}[\|\bar{\bsg}^0_{k}\|_{1+\gamma}^2]
\nonumber\\
&\le W_{0}+\frac{\kappa_4}{L_f^{2}}W_{4, 0}+(T+1)\mathcal{O}(n)\sigma^2\eta^2
+(T+1)\frac{\kappa_4}{L_f^{2}}L_f\sigma^2\eta.
\end{align}
Consider that $W_{0}+\frac{\kappa_4}{L_f^{2}}W_{4, 0}=\mathcal{O}(n/\eta)$, $\eta=\kappa_2/\beta$, and $\beta=\kappa_2\sqrt{T}/\sqrt{n}$, then we have \eqref{fo_pb:coro-sg-sm-equ3}.
\end{proof}
\section{Numerical Experiments}\label{fo_pb:exp}
\subsection{Two-layer Fully Connected Neural Networks}\label{fo:exp:nn} 
\subsubsection{Experiments Description}
Any neural network that has 2 or more than 2 layers is a non-convex optimization problem \cite{pmlr-v38-choromanska15}.
In this case, we consider the training of neural networks (NN) for image classification tasks of subset of the database MNIST \cite{lecun2010mnist} with 10 agents generated randomly following the Erd\H{o}s - R\' enyi model with the connection probability of $0.4$ shown in Figure~\ref{fig:comm}. 
The subset of MNIST contains 5000 images of 10 digits(0-9), half of them are used for training and the rest of them are used for testing.
\begin{figure}[!ht]
\centering
  \includegraphics[width=0.45\textwidth]{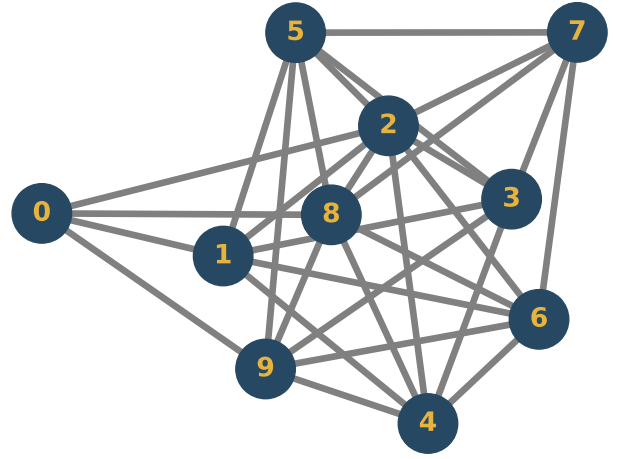}
  \caption{Communication topology of 10 nodes.}
  \label{fig:comm}
\end{figure}
For each agent, the training set has been divided evenly and is i.i.d, i.e. each agent has 250 samples. The neural network architecture consists of an input and output layer and 2 hidden layers.
Each local neural network architecture is shown in Fig.~\ref{fig:nn} in appendix~\ref{app:exp}.
Since the dimension of a single input image is $20\times 20 = 400$\footnote{In this experiment, we use a modified subset of the original MNIST dataset.}, the input layer contains $401$ neurons including a biased neuron.
Following the input layer, we construct a hidden layer with $51$ neurons including a biased neuron.
Followed by both the hidden layer and the output layer, we apply the logistic sigmoid function.
We demonstrate the result in terms of the empirical risk function \cite{bottou2012stochastic}, which is given as
\begin{align*}
R(\bm z) =& - \frac{1}{n}\sum_{i = 1}^{n}\frac{1}{m_{n}} \sum _{j=1}^{m_{n}} \sum _{k=0}^{9}\Big(t_{k} \ln y_k(\bm{x}, \bm z) \\ &+(1-t_{k})\ln (1 - y_k(\bm{x}, \bm z))\Big)
\end{align*}
where $m_{n}$ indicates the size of data set for each agent, $t_k$ denotes the target (ground truth) of digit $k$ corresponding to a single image, $\bm x$ is a single image input, $\bm z=(z^{(1)}, z^{(2)})$ where $z^{(1)}$ and $z^{(2)}$ are the weights for the 2 layers separately, and $y_k \in [0, 1]$ is the output which expresses the probability of digit $k = 0, \dots, 9$.
The mapping from input to output is given as:
\begin{align*}
y_k(\bm{x}, \bm z) = \text{Sig} \left( \sum_{j=0}^{50} z_{k, j}^{(2)} \text{Sig} \left( \sum_{i = 0}^{20 \times 20} z_{j, i}^{(1)}x_{i}\right)\right),
\end{align*}
where $\text{Sig} (s) = \frac{1}{1+\exp(-s)}$ is the logistic sigmoid function.
The parameters for the proposed and comparison algorithms are summarized in Table~\ref{tab:nn-par} in Appendix~\ref{app:exp}.
\subsubsection{Results}
We compare Algorithm~\ref{fopb:algorithm-random-pd} under time-varying parameters and fixed parameters with existing algorithms, such as DSGPA-T and DSGPA-F from \cite{yi2020primal}, D-SGD-1, the distributed SGD algorithm \cite{jiang2017collaborative,lian2017can}, D-SGD-2, the distributed SGD algorithm \cite{george2019distributed}, $\mathrm{D}^2$ \cite{Tang2018Decentralized}, DM-SGD the distributed momentum SGD algorithm \cite{Yu2019on}, D-SGT-1, the distributed stochastic gradient tracking algorithm   \cite{lu2019gnsd,xin2019distributed}, D-SGT-2, the distributed stochastic gradient tracking algorithm \cite{zhang2019decentralized,pu2018distributed}, and D-ASG, the distributed accelerated stochastic gradient methods \cite{fallah2019robust}.
In Figure~\ref{fig:risk}, we can see that DSGPA-T-PB has exactly the same convergence rate as the distributed momentum SGD algorithm \cite{Yu2019on} and D-ASG \cite{fallah2019robust}, which are accelerated algorithms.
Compared to the other existing non-accelerated distributed algorithms, DSGPA-T-PB converges faster.
In terms of testing accuracy shown in Table~\ref{tab:nn-acc}, the base line is set by centralized SGD algorithm, which is $93.00\%$.
All the distributed algorithms can achieve a good accuracy result ranging from $90.44\%$ ($\mathrm{D}^2$) to $93.44\%$ (the distributed momentum SGD algorithm)
DSGPA-F-PB has a better result than D-ASG but a slightly lower accuracy than the distributed momentum SGD algorithm by $0.733\%$.
In the experiment in Sec.~\ref{fo:exp:cnn}, we can see that DSGPA-F-PB outperforms much more than the distributed momentum SGD algorithm.
\begin{figure}[!ht]
\centering
  \includegraphics[width=0.45\textwidth]{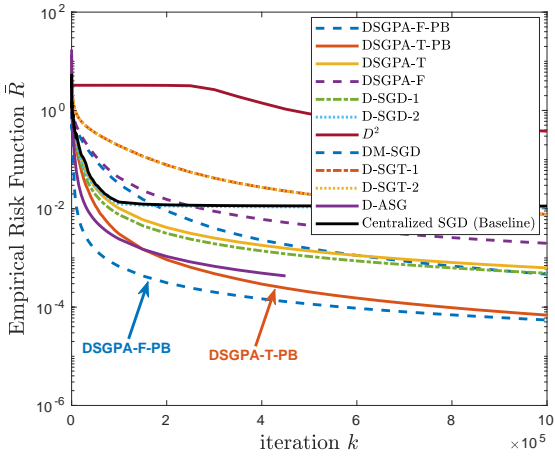}
  \caption{Training performance comparison.}
  \label{fig:risk}
\end{figure}
\begin{table}[!ht]
\caption{{Accuracy on each algorithm in NN experiment.} }
\label{tab:nn-acc}
\vskip 0.15in
\begin{center}
\begin{scriptsize}
\begin{tabular}{M{2.6cm}|M{1.2cm}}
\hline
Algorithm & Accuracy  \\
\hline
DSGPA-F-PB & $\mathbf{92.76\%}$\\
\hline
DSGPA-T-PB & $\mathbf{92.22\%}$\\
\hline
DSGPA-T & $93.04\%$\\
\hline
DSGPA-F & $92.76\%$\\
\hline
DM-SGD \cite{Yu2019on} & $93.44\%$\\
\hline
D-SGD-1 \cite{jiang2017collaborative,lian2017can}& $92.96\%$\\
\hline
D-SGD-2 \cite{george2019distributed} & $92.88\%$\\
\hline
$D^{2}$ \cite{Tang2018Decentralized}& $90.44\%$\\
\hline
D-SGT-1 \cite{lu2019gnsd,xin2019distributed} & $92.88\%$\\
\hline
D-SGT-2 \cite{zhang2019decentralized,pu2018distributed} & $92.96\%$\\
\hline
D-ASG \cite{fallah2019robust} & $90.68\%$\\
\hline
C-SGD & $93\%$\\
\hline
\end{tabular}
\end{scriptsize}
\end{center}
\vskip -0.1in
\end{table}
\subsection{Convolutional Neural Networks}\label{fo:exp:cnn}
\subsubsection{Experiments Description}
In this section, we consider a more complicated deep learning task of training a convolutional neural networks (CNN) model on the MNIST dataset.
We build a CNN model for each agent with five 3$\times$3 convolutional layers using ReLU as activation function, one average pooling layer with filters of size 2$\times$2, one sigmoid layer with dimension 360, another sigmoid layer with dimension 60, one softmax layer with dimension 10, shown in Fig.~\ref{fig:cnn} in appendix~\ref{app:exp}.
In this experiment, we use the entire MNIST data set. We use the same  communication graph as in the above NN  experiment. Each agent is assigned 6000 data points randomly. We set the batch size as 20, which means at each iteration, 20 data points are chosen by the agent to update the gradient, which also follows a uniform distribution. For each algorithm, we perform 10 epochs to train the CNN model.
We summarize the parameters for all compared algorithms in Table~\ref{tab:cnn-par} in Appendix~\ref{app:exp}.
\subsubsection{Results}
First, we compare Algorithm~\ref{fopb:algorithm-random-pd} under fix parameters with the fastest converging results from Sec.~\ref{fo:exp:nn}: the distributed momentum SGD algorithm \cite{Yu2019on}, D-ASG \cite{fallah2019robust}, DSGPA-T and DSGPA-F from \cite{yi2020primal}, distributed SGD algorithm \cite{jiang2017collaborative,lian2017can}, and C-SGD.
The training loss and testing accuracy are shown in Figure~\ref{fig:10_loss} and Figure~\ref{fig:10_acc} separately. Moreover, the testing accuracy is listed in Table~\ref{tab:cnn-acc}.
We can easily see that DSGPA-F-PB has the fastest convergence rate and highest accuracy among the all compared algorithms.
In Figure~\ref{fig:10_loss}, we can conclude that DSGPA-F-PB converges fast and has the lowest training loss after 10 epochs' training.
In Figure~\ref{fig:10_acc}, we can see that DSGPA-F-PB can achieve the highest testing accuracy $98.02\%$, which is also confirmed in Table~\ref{tab:cnn-acc}.
\begin{figure}[!ht]
\centering
        \includegraphics[width=0.45\textwidth]{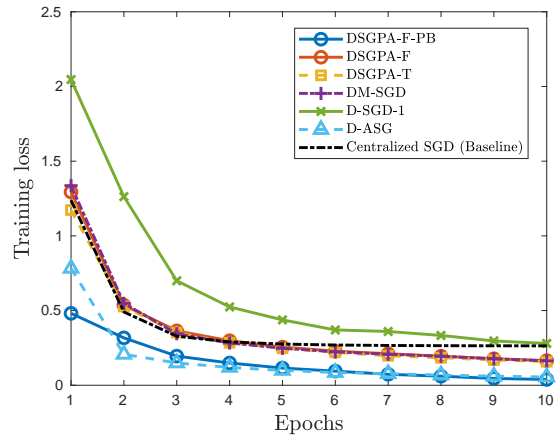}
        \caption{Training loss comparison.}
        \label{fig:10_loss}
\end{figure}

\begin{figure}[!ht]
\centering
        \includegraphics[width=0.45\textwidth]{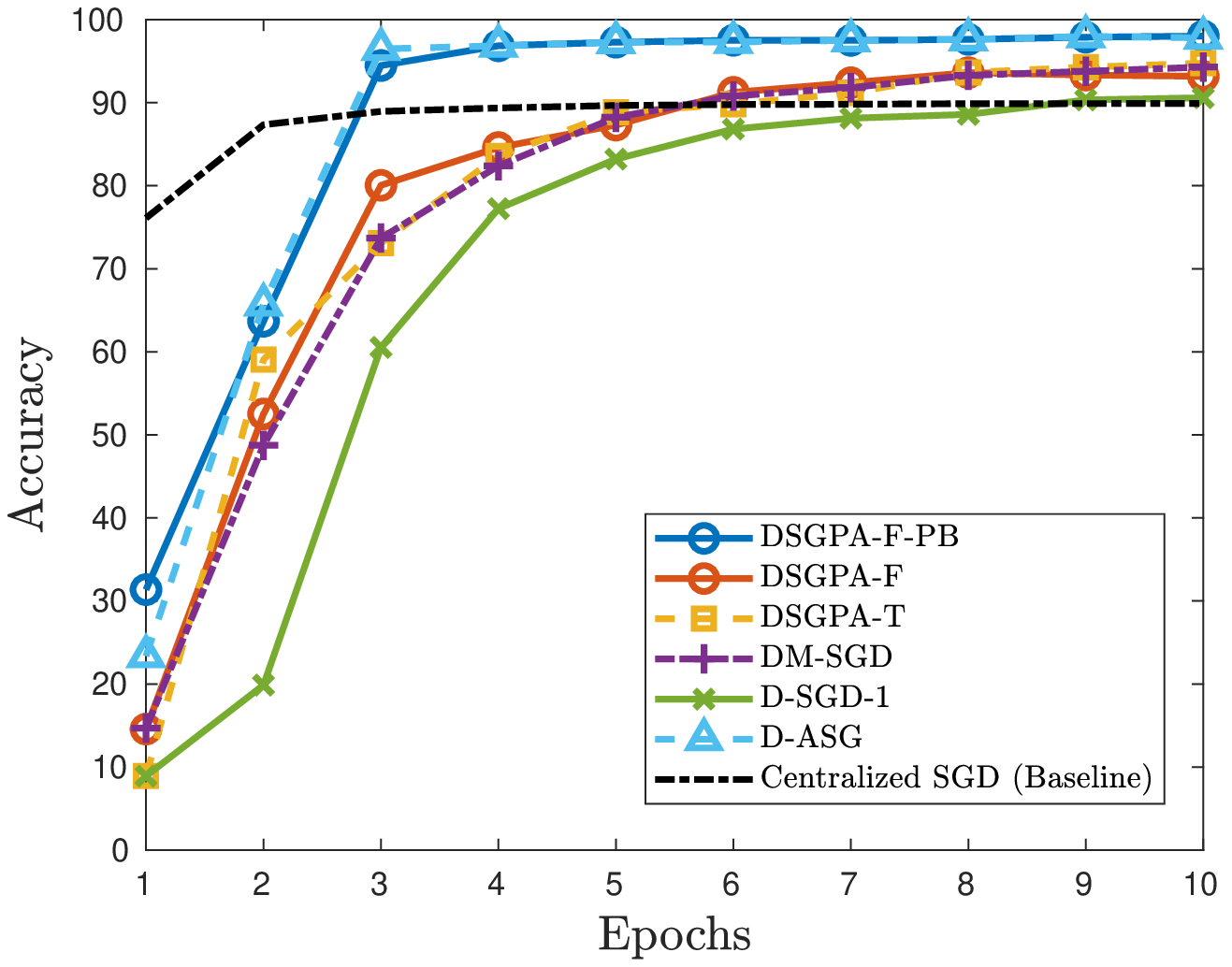}
        \caption{Testing accuracy comparison.}
        \label{fig:10_acc}
\end{figure}

\begin{table}[!ht]
\caption{{Accuracy on each algorithm in CNN experiment.} }
\label{tab:cnn-acc}
\vskip 0.15in
\begin{center}
\begin{scriptsize}
\begin{tabular}{M{2.1cm}|M{1.2cm}}
\hline
Algorithm & Accuracy  \\

\hline
DSGPA-F-PB & $\mathbf{98.02\%}$\\

\hline
DSGPA-T & $94.75\%$\\

\hline
DSGPA-F & $93.17\%$\\

\hline
DM-SGD \cite{Yu2019on} & $94.29\%$\\

\hline
D-ASG \cite{fallah2019robust} & $97.80\%$\\

\hline
D-SGD \cite{jiang2017collaborative,lian2017can}& $92.96\%$\\

\hline
C-SGD & $89.91\%$\\
\hline
\end{tabular}
\end{scriptsize}
\end{center}
\vskip -0.1in
\end{table}

\subsection{Robustness of $\gamma$}\label{fo_pb:exp:robust}
In order to show the effect of $\gamma$, we conduct another experiment based on the same CNN architecture but with 50 nodes.
The communication topology is shown in Figure~\ref{fig:comm_50}.
We tested several $\gamma \in \{0.3, 0.5, 0.7, 0.9, 1\}$. 
When $\gamma = 1$, the proposed algorithm is the same as DSGPA-F \cite{yi2020primal}, when $\gamma = 0$, the powerball gradient term $\sgn(\nabla)|\nabla|^{\gamma}$ becomes $\sgn (\nabla)$, which can be seen as a new distributed optimization algorithm based on \textit{signSGD}~\cite{pmlr-v80-bernstein18a}.
In this experiment, each agent is assigned $1,200$ samples, the network is trained for $20$ epochs, and the training loss and testing accuracy are shown in Figure~\ref{fig:gamma_50_node_loss} and Figure~\ref{fig:gamma_50_node_acc} respectively.
\begin{figure}[!ht]
\centering
  \includegraphics[width=0.45\textwidth]{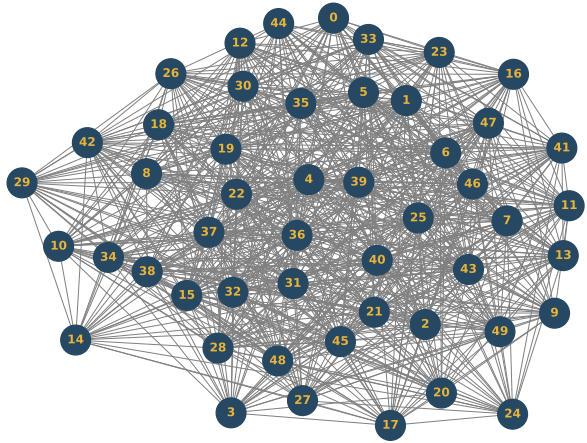}
  \caption{Communication topology of 50 nodes.}
  \label{fig:comm_50}
\end{figure}
\begin{figure}[!ht]
\centering
        \includegraphics[width=0.45\textwidth]{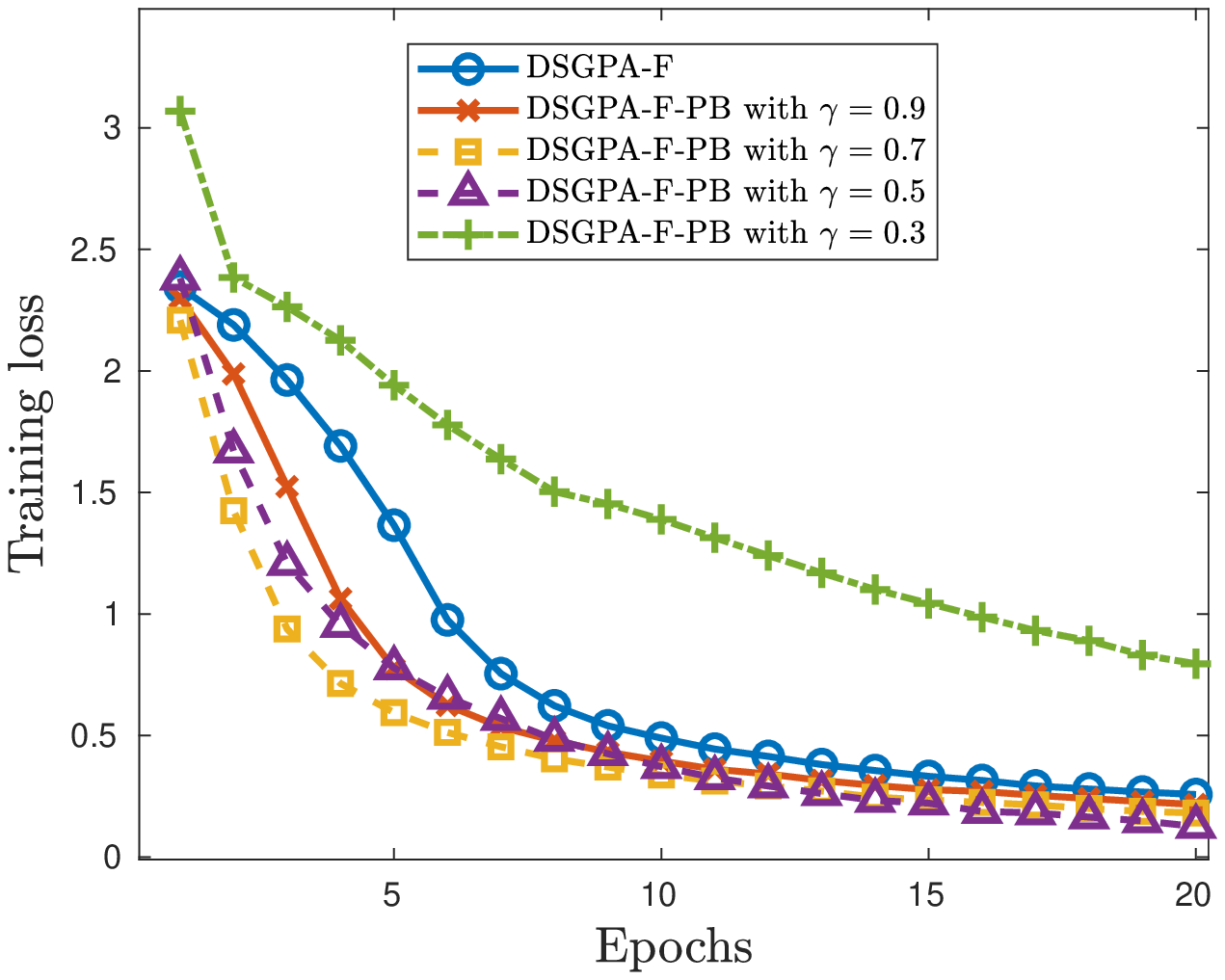}
        \caption{Training loss comparison.}
        \label{fig:gamma_50_node_loss}
\end{figure}

\begin{figure}[!ht]
\centering
        \includegraphics[width=0.45\textwidth]{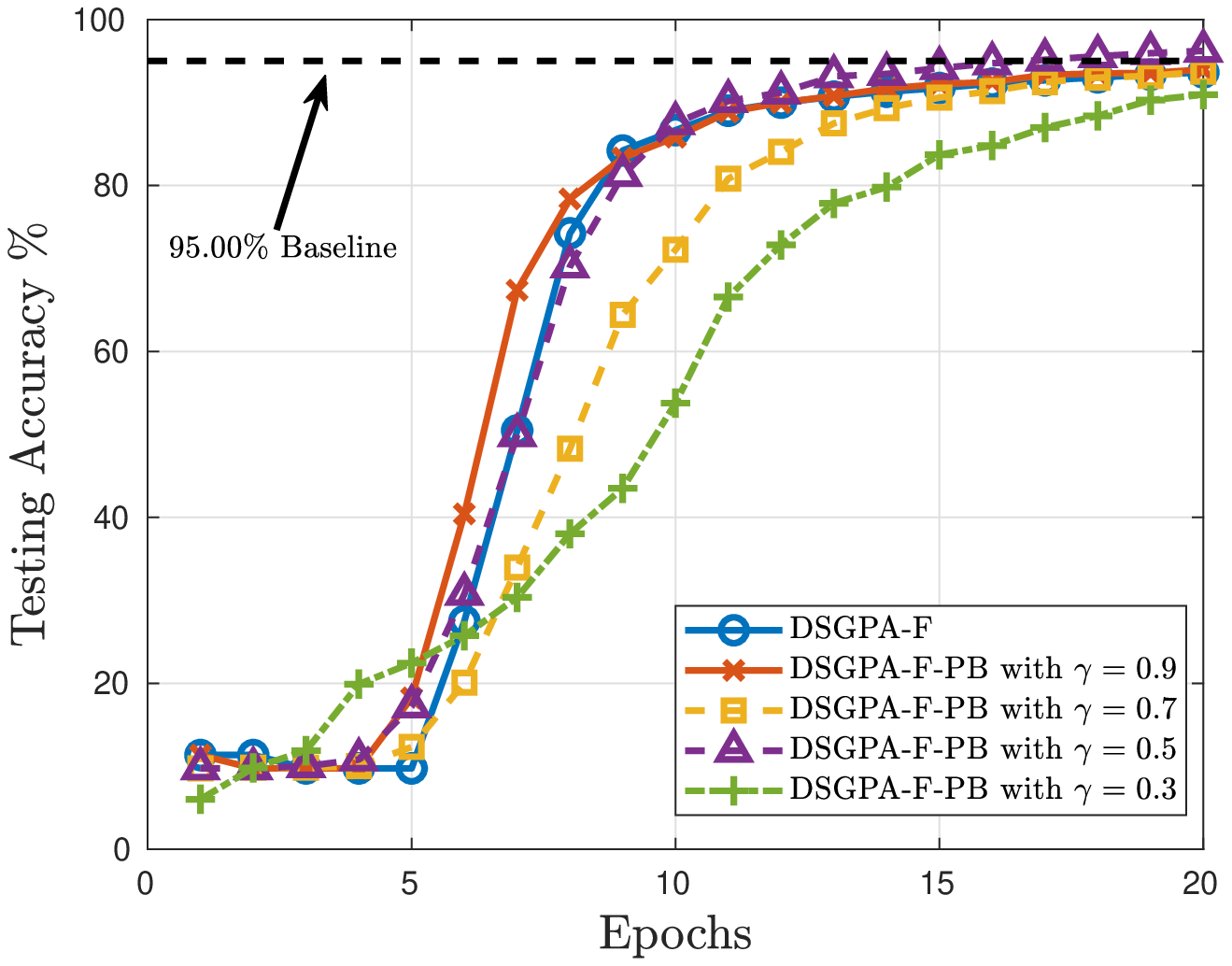}
        \caption{Testing accuracy comparison.}
        \label{fig:gamma_50_node_acc}
\end{figure}
From both experiments in Sec.~\ref{fo:exp:nn} and Sec.~\ref{fo:exp:cnn}, we can conclude that DSGPA-F-PB has the best performance overall in terms of the speed of convergence and testing accuracy, especially in the more complicated CNN case, which indicates that DSGPA-F-PB can be applied for training more complicated and complex CNN models in distributed and decentralized deep learning.
\section{Conclusion}\label{fo_pb:con}
In this paper, we examined how to accelerate the convergence of distributed optimization algorithms on general non-convex problems.
A distributed primal-dual stochastic gradient descent~(SGD) equipped with ``powerball'' method was proposed and proven to achieve the linear speedup convergence rate $\mathcal{O}(1/\sqrt{nT})$ for general smooth (possibly non-convex) cost functions.
We compared the proposed algorithm and state-of-the-art distributed SGD algorithms and centralized SGD algorithms in distributed machine learning experiments.
The experiments demonstrated the benefits of the proposed algorithm, which match the theoretical results and analysis.

\section*{Acknowledgments}
The authors would like to thank Dr. Xinlei Yi and Dr. Ye Yuan for their insightful inspirations and motivations on this work.

\bibliographystyle{IEEEtran}
\bibliography{fo_pb}
\begin{appendices}\label{sec:app}

\section{Proof of Lemma~\ref{fo_pb:lemma:lya}}\label{fo_pb:lyap_lemma}
\begin{proof}
Consider the following Lyapunov candidate function 
\begin{align}
W_{k}= &\underbrace{\frac{1}{2}\|\bsx_{k}\|^2_{\bsK}}_{W_{1, k}} + \underbrace{\frac{1}{2}\Big\|\bsv_k+\frac{1}{\beta}\bsg_k^0\Big\|^2_{\bsQ+\kappa_1\bsK}}_{W_{2, k}} \nonumber \\ &+ \underbrace{\bsx_k^\top\bsK\Big(\bm{v}_k+\frac{1}{\beta}\bsg_k^0\Big)}_{W_{3, k}} + \underbrace{n(f(\bar{x}_k)-f^*)}_{W_{4, k}}\label{fo_pb:lyap_def_k}
\end{align}
where $\bsQ=R\Lambda^{-1}_1R^{\top}\otimes {\bm I}_p$.
Additionally, we denote $\bsg_k=n \nabla f({\bsx}_{k})$, $\bar{\bsg}_k=\bsH\bsg_{k}$, and $\bar{\bsg}_k^u=\bsH\bsg_k^u$.

(i) We have
\begin{align}
&\mathbb{E}[W_{1,k+1}]
=\mathbb{E}\Big[\frac{1}{2}\|\bm{x}_{k+1} \|^2_{\bsK}\Big]\nonumber\\
&\overset{\mathrm{Eq.~\ref{fopb:alg:random-pd-x}}}{=\joinrel=}\mathbb{E}\Big[\frac{1}{2}\|\bm{x}_k-\eta(\alpha\bsL\bm{x}_k+\beta\bm{v}_k+\sigma(\bsg^u_k, \gamma)) \|^2_{\bsK}\Big]\nonumber\\
&\overset{\text{(a)}}{=}\mathbb{E}\Big[\frac{1}{2}\|\bm{x}_k\|^2_{\bsK}-\eta\alpha\|\bsx_k\|^2_{\bsL}
+\frac{1}{2}\eta^2\alpha^2\|\bsx_k\|^2_{\bsL^2}
\nonumber\\
&~~~-\eta\beta\bsx^\top_k({\bm I}_{np}-\eta\alpha\bsL)\bsK\Big(\bm{v}_k+\frac{1}{\beta}\sigma(\bsg^u_k, \gamma)\Big)\nonumber\\
&~~~+\frac{1}{2}\eta^2\beta^2\Big\|\bm{v}_k+\frac{1}{\beta}\sigma(\bsg^u_k, \gamma)\Big\|^2_{\bsK}\Big]\nonumber\\
&\overset{\text{(b)}}{=}W_{1,k}-\|\bsx_k\|^2_{\eta\alpha\bsL
-\frac{1}{2}\eta^2\alpha^2\bsL^2}\nonumber\\
&~~~-\eta\beta\bsx^\top_k({\bm I}_{np}-\eta\alpha\bsL)\bsK\Big(\bm{v}_k
+\frac{1}{\beta}\bsg_k\Big)\nonumber\\
&~~~+\frac{1}{2}\eta^2\beta^2\mathbb{E}\Big[\Big\|\bm{v}_k+\frac{1}{\beta}\bsg_k^0
+\frac{1}{\beta}\sigma(\bsg^u_k, \gamma)-\frac{1}{\beta}\bsg_k^0\Big\|^2_{\bsK}\Big]\nonumber\\
&\overset{\text{(c)}}{\le} W_{1,k}-\|\bsx_k\|^2_{\eta\alpha\bsL
-\frac{1}{2}\eta^2\alpha^2\bsL^2}\nonumber\\
&~~~-\eta\beta\bsx^\top_k\bsK\Big(\bm{v}_k+\frac{1}{\beta}\bsg_k^0\Big)\nonumber\\
&~~~+\frac{1}{2}\eta\|\bm{x}_k\|^2_{\bsK}
+\frac{1}{2}\eta\|\bsg_k-\bsg_k^0\|^2\nonumber\\
&~~~+\frac{1}{2}\eta^2\alpha^2\|\bm{x}_k\|^2_{\bsL^2}
+\frac{1}{2}\eta^2\beta^2\Big\|\bm{v}_k+\frac{1}{\beta}\bsg_k^0\Big\|^2_{\bsK}\nonumber\\
&~~~+\frac{1}{2}\eta^2\alpha^2\|\bm{x}_k\|^2_{\bsL^2}
+\frac{1}{2}\eta^2\|\bsg_k-\bsg_k^0\|^2\nonumber\\
&~~~+\eta^2\beta^2\Big\|\bm{v}_k+\frac{1}{\beta}\bsg_k^0\Big\|^2_{\bsK}
+\eta^2\mathbb{E}[\|\sigma(\bsg^u_k, \gamma)-\bsg_k^0\|^2]\nonumber\\
&\overset{\text{(d)}}{\le} W_{1,k}-\|\bsx_k\|^2_{\eta\alpha\bsL-\frac{1}{2}\eta\bsK
-\frac{3}{2}\eta^2\alpha^2\bsL^2-\eta(1+5\eta)L_f^2\bsK}
\nonumber\\
&~~~-\eta\beta\bsx^\top_k\bsK\Big(\bm{v}_k+\frac{1}{\beta}\bsg_k^0\Big)
+\Big\|\bm{v}_k+\frac{1}{\beta}\bsg_k^0\Big\|^2_{\frac{3}{2}\eta^2\beta^2\bsK}\nonumber\\
&~~~+2n\sigma^{2}\eta^{2},\label{fo_pb:v1k}
\end{align}
where (a) holds due to Lemma~1 and~2 in~\cite{Yi2018distributed}; (b) holds due to $\mathbb{E}[\bsg^u_k] = \bsg_k$ and that $x_{i,k}$ and $v_{i,k}$ are independent of $u_{i,k}$ and $\xi_{i,k}$; (c) holds due to the Cauchy--Schwarz inequality, $\rho(\bsK)=1$ and $1+\gamma \leq 2$; (d) holds due to $\|\bsg^0_{k}-\bsg_{k}\|^2\le L_f^2\|\bar{\bsx}_{k}-\bsx_{k}\|^2=L_f^2\|\bsx_{k}\|^2_{\bsK}$, $\mathbb{E}[\|\sigma(\bsg^u_k, \gamma)-\bsg_k\|^2]\le \mathbb{E}[\|\bsg^u_k-\bsg_k\|^2]\le n\sigma^2$.


(ii)
\begin{align}
W_{2,k+1}
=\frac{1}{2}\Big\|\bsv_{k+1}+\frac{1}{\beta}\bsg_{k+1}^0\Big\|^2_{\bsQ+\kappa_1\bsK}\label{fo_pb:v2k-1}
\end{align}
\begin{align}
&\overset{\mathrm{Eq.~\eqref{fopb:alg:random-pd-q}}}{=\joinrel=}\frac{1}{2}\Big\|\bm{v}_k+\frac{1}{\beta}\bsg_{k}^0+\eta\beta\bsL\bm{x}_k
+\frac{1}{\beta}(\bsg_{k+1}^0-\bsg_{k}^0) \Big\|^2_{\bsQ+\kappa_1\bsK}\nonumber\\
&\overset{\text{(e)}}{=}W_{2,k}+\eta\beta\bsx^\top_k(\bsK+\kappa_1\bsL)\Big(\bm{v}_k+\frac{1}{\beta}\bsg_k^0\Big)\nonumber\\
&~~~+\|\bsx_k\|^2_{\frac{1}{2}\eta^2\beta^2(\bsL+\kappa_1\bsL^2)}
+\frac{1}{2\beta^2}\Big\|\bsg_{k+1}^0-\bsg_{k}^0\Big\|^2_{\bsQ+\kappa_1\bsK}\nonumber\\
&~~~+\frac{1}{\beta}\Big(\bm{v}_k+\frac{1}{\beta}\bsg_{k}^0
+\eta\beta\bsL\bm{x}_k\Big)^\top(\bsQ
+\kappa_1\bsK)(\bsg_{k+1}^0-\bsg_{k}^0)\nonumber\\
&\overset{\text{(f)}}{\le} W_{2,k}+\eta\beta\bsx^\top_k(\bsK+\kappa_1\bsL)\Big(\bm{v}_k+\frac{1}{\beta}\bsg_k^0\Big)\nonumber\\
&~~~+\|\bsx_k\|^2_{\frac{1}{2}\eta^2\beta^2(\bsL+\kappa_1\bsL^2)}
+\frac{1}{2\beta^2}\|\bsg_{k+1}^0-\bsg_{k}^0\|^2_{\bsQ+\kappa_1\bsK}\nonumber\\
&~~~+\frac{\eta}{2}\Big\|\bm{v}_k+\frac{1}{\beta}\bsg_{k}^0\Big\|^2_{\bsQ+\kappa_1\bsK}
+\frac{1}{2\eta\beta^2}\|\bsg_{k+1}^0-\bsg_{k}^0\|^2_{\bsQ+\kappa_1\bsK}\nonumber\\
&~~~+\frac{1}{2}\eta^2\beta^2\|\bsL\bm{x}_k\|^2_{\bsQ+\kappa_1\bsK}
+\frac{1}{2\beta^2}\|\bsg_{k+1}^0-\bsg_{k}^0\|^2_{\bsQ+\kappa_1\bsK}\nonumber\\
&\overset{\text{(g)}}{=} W_{2,k}+\eta\beta\bsx^\top_k(\bsK+\kappa_1\bsL)\Big(\bm{v}_k+\frac{1}{\beta}\bsg_k^0\Big)\nonumber\\
&~~~+\|\bsx_k\|^2_{\eta^2\beta^2(\bsL+\kappa_1\bsL^2)}
+\Big\|\bm{v}_k+\frac{1}{\beta}\bsg_{k}^0
\Big\|^2_{\frac{1}{2}\eta(\bsQ+\kappa_1\bsK)}\nonumber\\
&~~~+\frac{1}{\beta^2}\Big(1+\frac{1}{2\eta}\Big)
\|\bsg_{k+1}^0-\bsg_{k}^0\|^2_{\bsQ+\kappa_1\bsK}\nonumber\\
&\overset{\text{(h)}}{\le} W_{2,k}+\eta\beta\bsx^\top_k(\bsK+\kappa_1\bsL)\Big(\bm{v}_k+\frac{1}{\beta}\bsg_k^0\Big)\nonumber\\
&~~~+\|\bsx_k\|^2_{\eta^2\beta^2(\bsL+\kappa_1\bsL^2)}
+\Big\|\bm{v}_k+\frac{1}{\beta}\bsg_{k}^0
\Big\|^2_{\frac{1}{2}\eta(\bsQ+\kappa_1\bsK)}\nonumber\\
&~~~+\frac{1}{\beta^2}\Big(1+\frac{1}{2\eta}\Big)\Big(\frac{1}{\rho_2(L)}+\kappa_1\Big)
\|\bsg_{k+1}^0-\bsg_{k}^0\|^2\nonumber\\
&\overset{\text{(i)}}{\le} W_{2,k}+\eta\beta\bsx^\top_k(\bsK+\kappa_1\bsL)\Big(\bm{v}_k+\frac{1}{\beta}\bsg_k^0\Big)\nonumber\\
&~~~+\|\bsx_k\|^2_{\eta^2\beta^2(\bsL+\kappa_1\bsL^2)}
+\Big\|\bm{v}_k+\frac{1}{\beta}\bsg_{k}^0
\Big\|^2_{\frac{1}{2}\eta(\bsQ+\kappa_1\bsK)}\nonumber\\
&~~~+\frac{\eta}{\beta^2}\Big(\eta+\frac{1}{2}\Big)
\Big(\frac{1}{\rho_2(L)}+\kappa_1\Big)L_f^2\|\bar{\bsg}^u_{k}\|^2,\label{fo_pb:v2k-2}
\end{align}
where (e) holds due to (a) holds due to Lemma~1 and~2 in~\cite{Yi2018distributed}; (f) holds due to the Cauchy--Schwarz inequality; (g) holds due to Lemma~1 and~2 in~\cite{Yi2018distributed}; (h) holds due to $\rho(\bsQ+\kappa_1\bsK)\le\rho(\bsQ)+\kappa_1\rho(\bsK)$, and $\rho(\bsK)=1$; (i) holds due to $\|\bsg^0_{k+1}-\bsg^0_{k}\|^2\le \eta^2L_f^2\|\bar{\bsg}^u_{k}\|^2
\le\eta^2L_f^2\|\bsg^u_{k}\|^2$.
Moreover, we have the following two inequalities hold:
\begin{equation}\label{fo_pb:v2k-3}
\|\bsg_{k+1}^0\|^2_{\bsQ+\kappa_1\bsK}
\le\Big(\frac{1}{\rho_2(L)}+\kappa_1\Big)\|\bsg_{k+1}^0\|^2.
\end{equation}
\begin{align}\label{fo_pb:v2k-4}
\Big\|\bm{v}_k+\frac{1}{\beta_k}\bsg_{k}^0\Big\|^2_{\bsQ+\kappa_1\bsK}
\le\Big(\frac{1}{\rho_2(L)}+\kappa_1\Big)\Big\|\bm{v}_k+\frac{1}{\beta_k}\bsg_{k}^0
\Big\|^2_{\bsK}.
\end{align}
Then, from \eqref{fo_pb:v2k-1}--\eqref{fo_pb:v2k-4}, we have
\begin{align}
&W_{2,k+1}\nonumber\\
&\le W_{2,k}
+\eta\beta\bsx^\top_k(\bsK+\kappa_1\bsL)\Big(\bm{v}_k+\frac{1}{\beta}\bsg_k^0\Big)\nonumber\\
&~~~+\frac{1}{2}\eta\Big(\frac{1}{\rho_2(L)}+\kappa_1\Big)
\Big\|\bm{v}_k+\frac{1}{\beta}\bsg_{k}^0\Big\|^2_{\bsK}\nonumber\\
&~~~+\|\bsx_k\|^2_{\eta^2\beta^2(\bsL+\kappa_1\bsL^2)}\nonumber\\
&~~~+\frac{\eta}{\beta^2}\Big(\eta+\frac{1}{2}\Big)
\Big(\frac{1}{\rho_2(L)}+\kappa_1\Big)L_f^2\|\bar{\bsg}^u_{k}\|^2.
\label{fo_pb:v2k}
\end{align}

(iii) We have
\begin{align}
W_{3,k+1}=\bsx_{k+1}^\top\bsK\Big(\bm{v}_{k+1}+\frac{1}{\beta}\bsg_{k+1}^0\Big).\label{fo_pb:v3k-1}
\end{align}
\begin{align}
&\mathbb{E}\Big[W_{3,k+1}\Big]\nonumber\\
&\overset{\mathrm{Eq.~\eqref{fopb:alg:random-pd}}}{=\joinrel=}\mathbb{E}\Big[(\bm{x}_k-\eta(\alpha\bsL\bm{x}_k+\beta\bm{v}_k
+\bsg_k^0+\sigma(\bsg^u_k, \gamma)-\bsg_k^0))^\top
\nonumber\\
&~~~\bsK\Big(\bm{v}_k+\frac{1}{\beta}\bsg_{k}^0+\eta\beta\bsL\bm{x}_k
+\frac{1}{\beta}(\bsg_{k+1}^0-\bsg_{k}^0)\Big)\Big]\nonumber\\
&\overset{\text{(j)}}{=}\bm{x}_k^\top(\bsK-\eta(\alpha+\eta\beta^2)\bsL)\Big(\bm{v}_k+\frac{1}{\beta}\bsg_{k}^0\Big)+\|\bm{x}_k\|^2_{\eta\beta(\bsL-\eta\alpha\bsL^2)}\nonumber\\
&~~~+\frac{1}{\beta}\bm{x}_k^\top(\bsK-\eta\alpha\bsL)
\mathbb{E}[\bsg_{k+1}^0-\bsg_{k}^0]-\eta\beta\Big\|\bm{v}_k+\frac{1}{\beta}\bsg_{k}^0\Big\|^2_{\bsK}\nonumber\\
&~~~-\eta\Big(\bm{v}_k+\frac{1}{\beta}\bsg_{k}^0\Big)^\top\bsK
\mathbb{E}[\bsg_{k+1}^0-\bsg_{k}^0]\nonumber\\
&~~~-\eta(\bsg_k-\bsg_k^0)^\top
\bsK\Big(\bm{v}_k+\frac{1}{\beta}\bsg_{k}^0+\eta\beta\bsL\bm{x}_k\Big)\nonumber\\
&~~~-\frac{1}{\beta}\mathbb{E}[\eta(\sigma(\bsg^u_k, \gamma)-\bsg_k^0)^\top
\bsK(\bsg_{k+1}^0-\bsg_{k}^0)]\nonumber\\
&\overset{\text{(k)}}{\le}\bm{x}_k^\top(\bsK-\eta\alpha\bsL)\Big(\bm{v}_k+\frac{1}{\beta}\bsg_{k}^0\Big)
+\frac{1}{2}\eta^2\beta^2\|\bsL\bsx_k\|^2\nonumber\\
&~~~+\frac{1}{2}\eta^2\beta^2\Big\|\bm{v}_k+\frac{1}{\beta}\bsg_{k}^0\Big\|^2_{\bsK}
+\|\bm{x}_k\|^2_{\eta\beta(\bsL-\eta\alpha\bsL^2)}\nonumber\\
&~~~+\frac{1}{2}\eta\|\bm{x}_k\|^2_\bsK
+\frac{1}{2\eta\beta^2}\mathbb{E}[\|\bsg_{k+1}^0-\bsg_{k}^0\|^2\nonumber\\
&~~~+\frac{1}{2}\eta^2\alpha^2\|\bsL\bm{x}_k\|^2
+\frac{1}{2\beta^2}\mathbb{E}[\|\bsg_{k+1}^0-\bsg_{k}^0\|^2]\nonumber\\
&~~~
-\eta\beta\Big\|\bm{v}_k+\frac{1}{\beta}\bsg_{k}^0\Big\|^2_{\bsK}\nonumber\\
&~~~+\frac{1}{2}\eta^2\beta^2\Big\|\bm{v}_k+\frac{1}{\beta}\bsg_{k}^0\Big\|^2_{\bsK}
+\frac{1}{2\beta^2}\mathbb{E}[\|\bsg_{k+1}^0-\bsg_{k}^0\|^2]\nonumber\\
&~~~+\frac{1}{2}\eta\|\bsg_k-\bsg_k^0\|^2
+\frac{1}{2}\eta\Big\|\bm{v}_k+\frac{1}{\beta}\bsg_{k}^0\Big\|^2_{\bsK}\nonumber\\
&~~~
+\frac{1}{2}\eta^2\|\bsg_k-\bsg_k^0\|^2
+\frac{1}{2}\eta^2\beta^2\|\bsL\bm{x}_k\|^2\nonumber\\
&~~~
+\frac{1}{2}\eta^2\mathbb{E}[\|\sigma(\bsg^u_k, \gamma)-\bsg_k^0\|^2]
+\frac{1}{2\beta^2}\mathbb{E}[\|\bsg_{k+1}^0-\bsg_{k}^0\|^2]\nonumber\\
&\overset{\text{(l)}}{\le} W_{3,k}
-\eta\alpha\bm{x}_k^\top\bsL\Big(\bm{v}_k+\frac{1}{\beta}\bsg_{k}^0\Big)\nonumber\\
&~~~+\|\bm{x}_k\|^2_{\eta(\beta\bsL+\frac{1}{2}\bsK)
+\eta^2(\frac{1}{2}\alpha^2-\alpha\beta+\beta^2)\bsL^2
+\eta(1+3\eta)L_f^2\bsK}\nonumber\\
&~~~+\eta^2\Big[1 + (\frac{1}{2\eta\beta^2}+\frac{3}{2\beta^2}\Big)L_f^2\Big]\mathbb{E}[\|\sigma(\bsg^u_k, \gamma)\|^2]\nonumber\\
&~~~+n\sigma^2\eta-\Big\|\bm{v}_k+\frac{1}{\beta}\bsg_{k}^0\Big\|^2_{\eta(\beta-\frac{1}{2}
-\eta\beta^2)\bsK}.
\label{fo_pb:v3k-2}
\end{align}
where (j) holds since $K_nL=LK_n=L$, $\mathbb{E}[\bsg^e_k] = \bsg^s_k$, and that $x_{i,k}$ and $v_{i,k}$ are independent; (k) holds due to the Cauchy--Schwarz inequality, the Jensen's inequality, and $\rho(\bsK)=1$; (l) holds due to holds due to $\|\bsg^0_{k}-\bsg_{k}\|^2\le L_f^2\|\bar{\bsx}_{k}-\bsx_{k}\|^2=L_f^2\|\bsx_{k}\|^2_{\bsK}$, $\mathbb{E}[\|\sigma(\bsg^u_k, \gamma)-\bsg_k\|^2]\le \mathbb{E}[\|\bsg^u_k-\bsg_k\|^2]\le n\sigma^2$, and $\|\bsg^0_{k+1}-\bsg^0_{k}\|^2\le \eta^2L_f^2\|\bar{\bsg}^u_{k}\|^2
\le\eta^2L_f^2\|\bsg^u_{k}\|^2$.

(iv) We have
\begin{align}
&\mathbb{E}[W_{4,k+1}]=\mathbb{E}[n(f(\bar{x}_{k+1})-f^*)]
=\mathbb{E}[\sum_{i=1}^{n}f_{i}(\bar{\bsx}_{k+1})-nf^*]\nonumber\\
&=\mathbb{E}[\sum_{i=1}^{n}f_{i}(\bar{\bsx}_k)-nf^*+\sum_{i=1}^{n}f_{i}(\bar{\bsx}_{k+1})
-\sum_{i=1}^{n}f_{i}(\bar{\bsx}_k)]\nonumber\\
&\overset{\text{(m)}}{\le}\mathbb{E}[\sum_{i=1}^{n}f_{i}(\bar{\bsx}_k)-nf^*
-\eta(\bar{\bsg}_{k}^u)^\top\bsg^0_k
+\frac{1}{2}\eta^2L_f\|\bar{\bsg}_{k}^u\|^2]\nonumber\\
&\overset{\text{(n)}}{=}W_{4,k}
-\eta(\bar{\bsg}_{k})^\top\bar{\bsg}^0_k
+\frac{1}{2}\eta^2L_f\mathbb{E}[\|\bar{\bsg}_{k}^u\|^2]\nonumber\\
&\overset{\text{(o)}}{=}W_{4,k}
-\frac{1}{2}\eta(\bar{\bsg}_{k})^\top(\bar{\bsg}_k+\bar{\bsg}^0_k-\bar{\bsg}_k)\nonumber\\
&~~~-\frac{1}{2}\eta(\bar{\bsg}_{k}-\bar{\bsg}^0_k+\bar{\bsg}^0_k)^\top\bar{\bsg}^0_k
+\frac{1}{2}\eta^2L_f\mathbb{E}[\|\bar{\bsg}_{k}^u\|^2]\nonumber\\
&\overset{\text{(p)}}{\le} W_{4,k}-\frac{1}{4}\eta(\|\bar{\bsg}_{k}\|^2
-\|\bar{\bsg}^0_k-\bar{\bsg}_k\|^2+\|\bar{\bsg}_{k}^0\|^2\nonumber\\
&~~~-\|\bar{\bsg}^0_k-\bar{\bsg}_k\|^2)
+\frac{1}{2}\eta^2L_f\mathbb{E}[\|\bar{\bsg}_{k}^u\|^2]\nonumber\\
&\overset{\text{(q)}}{\le} W_{4,k}-\frac{1}{4}\eta\|\bar{\bsg}_{k}\|^2
+\|\bsx_k\|^2_{\eta L_f^2\bsK}\nonumber\\
&~~~-\frac{1}{4}\eta\|\bar{\bsg}_{k}^0\|^2
+\frac{1}{2}\eta^2L_f\mathbb{E}[\|\bar{\bsg}^u_{k}\|^2],\label{fo_pb:v4k}
\end{align}
where (m) holds since that $\sum_{i=1}^{n}f_{i}$ is smooth under Assumption~\ref{fo_pb:ass:local}; (n) holds due to $\mathbb{E}[\bsg^e_k] = \bsg^s_k$, $x_{i,k}$ and $v_{i,k}$ are independent; (o) holds due to $(\bar{\bsg}_{k}^s)^\top\bsg^0_k=(\bsg_{k}^s)^\top\bsH\bsg^0_k=(\bsg_{k}^s)^\top\bsH\bsH\bsg^0_k
=(\bar{\bsg}_{k}^s)^\top\bar{\bsg}^0_k$; (p) holds due to the Cauchy--Schwarz inequality; and (q) holds due to  $\|\bar{\bsg}^0_k-\bar{\bsg}_k\|^2=\|\bsH(\bsg^0_{k}-\bsg_{k})\|^2\le\|\bsg^0_{k}-\bsg_{k}\|^2\le L_f^2\|\bsx_k\|^2_{\bsK}$, and $\mathbb{E}[\bar{\bsg}^u_k]=\mathbb{E}[\bsH\bsg^u_k]
=\bsH\mathbb{E}[\bsg^u_k]=\bar{\bsg}_k$.

Additionally, we have the following expression,
\begin{align}
&\mathbb{E}[\|\bar{\bsg}^u_k\|^2]
=\mathbb{E}[\|\bar{\bsg}^u_k-\bar{\bsg}_k+\bar{\bsg}_k\|^2]\nonumber\\
&\overset{\text{(r)}}{\le}2\mathbb{E}[\|\bar{\bsg}^u_k-\bar{\bsg}_k\|^2]+2\|\bar{\bsg}_k\|^2\nonumber\\
&=2n\mathbb{E}[\|\frac{1}{n}\sum_{i=1}^{n}(g^u_{i,k}-g_{i,k})\|^2]
+2\|\bar{\bsg}_k\|^2\nonumber\\
&=\frac{2}{n}\mathbb{E}[\|\sum_{i=1}^{n}(g^u_{i,k}-g_{i,k})\|^2]
+2\|\bar{\bsg}_k\|^2\nonumber\\
&=\frac{2}{n}\sum_{i=1}^{n}\mathbb{E}[\|g^u_{i,k}-g_{i,k}\|^2]
+2\|\bar{\bsg}_k\|^2\nonumber\\
&\overset{\text{(s)}}{\le}2\sigma^2+2\|\bar{\bsg}_k\|^2,\label{fo_pb:vkLya-3}
\end{align}
where (r) holds due to the Cauchy--Schwarz inequality; the last equality holds since $\{g^u_{i,k},~i\in[n]\}$ are independent of each other as assumed in Assumption~\ref{fo_pb:ass:stochastic-grad:xi}, $\bsx_{k}$ and $\bsv_{k}$ are independent, and $\mathbb{E}[g^u_{i,k}]=g_{i,k}$ as assumed in Assumption~\ref{fo_pb:ass:stochastic-grad:mean}; (s) holds due to $\mathbb{E}[\|\bsg^u_k-\bsg_k\|^2]\le n\sigma^2$.
From~\eqref{fo_pb:v4k} and~\eqref{fo_pb:vkLya-3}, we have 
\begin{align}
&\mathbb{E}[W_{4,k+1}] \le W_{4,k} + \|\bsx_k\|^2_{\frac{1}{2}\eta L_f^2\bsK} + L_f^{2}\sigma^{2}\eta^{2},
\end{align}
which is~\eqref{fo_pb:sgproof-vkLya2T_W4_og}.

(v) 
Recall the definition of $W_{k}$ in~\eqref{fo_pb:lyap_def_k} and combine~\eqref{fo_pb:v1k},~\eqref{fo_pb:v2k},~\eqref{fo_pb:v3k-2},~\eqref{fo_pb:vkLya-3}, then we have the following inequality directly,
\begin{align}
&\mathbb{E}[W_{k+1}]\nonumber\\
&\le W_{k}-\|\bsx_k\|^2_{\eta\alpha\bsL-\frac{1}{2}\eta\bsK
-\frac{3}{2}\eta^2\alpha^2\bsL^2-\eta(1+5\eta)L_f^2\bsK}
\nonumber\\
&~~~-\eta\beta\bsx^\top_k\bsK\Big(\bm{v}_k+\frac{1}{\beta}\bsg_k^0\Big)
+\Big\|\bm{v}_k+\frac{1}{\beta}\bsg_k^0\Big\|^2_{\frac{3}{2}\eta^2\beta^2\bsK}\nonumber\\
&~~~+2n\sigma^{2}\eta^{2}+\eta\beta\bsx^\top_k(\bsK+\kappa_1\bsL)\Big(\bm{v}_k+\frac{1}{\beta}\bsg_k^0\Big)\nonumber\\
&~~~+\frac{1}{2}\eta\Big(\frac{1}{\rho_2(L)}+\kappa_1\Big)
\Big\|\bm{v}_k+\frac{1}{\beta}\bsg_{k}^0\Big\|^2_{\bsK}\nonumber\\
&~~~+\|\bsx_k\|^2_{\eta^2\beta^2(\bsL+\kappa_1\bsL^2)}\nonumber\\
&~~~+\frac{\eta}{\beta^2}\Big(\eta+\frac{1}{2}\Big)
\Big(\frac{1}{\rho_2(L)}+\kappa_1\Big)L_f^2\|\bar{\bsg}^u_{k}\|^2\nonumber\\
&~~~-\eta\alpha\bm{x}_k^\top\bsL\Big(\bm{v}_k+\frac{1}{\beta}\bsg_{k}^0\Big)\nonumber\\
&~~~+\|\bm{x}_k\|^2_{\eta(\beta\bsL+\frac{1}{2}\bsK)
+\eta^2(\frac{1}{2}\alpha^2-\alpha\beta+\beta^2)\bsL^2
+\eta(1+3\eta)L_f^2\bsK}\nonumber\\
&~~~+\eta^2\Big[1 + (\frac{1}{2\eta\beta^2}+\frac{3}{2\beta^2}\Big)L_f^2\Big]\mathbb{E}[\|\sigma(\bsg^u_k, \gamma)\|^2]\nonumber\\
&~~~+n\sigma^2\eta-\Big\|\bm{v}_k+\frac{1}{\beta}\bsg_{k}^0\Big\|^2_{\eta(\beta-\frac{1}{2}
-\eta\beta^2)\bsK}\nonumber\\
&\overset{\text{(t)}}{\le} W_{k}-\|\bsx_k\|^2_{\eta\bsM_{1}-\eta^2\bsM_{2}}
-\Big\|\bm{v}_k+\frac{1}{\beta}\bsg_{k}^0\Big\|^2_{b_{1}\bsK}\nonumber\\
&~~~-b_{2}\eta \|\bar{\bsg}_{k}\|^2
-\frac{1}{4}\eta\|\bar{\bsg}^0_{k}\|_{1+\gamma}^2+b_{3}\sigma^2\eta^2+3n\sigma^2\eta^2,
\label{fo_pb:vkLya}
\end{align}
where (t) holds due to Lemma~\ref{fo_pb:lemma:pb}, $\alpha=\kappa_1\beta$,  $\eta=\frac{\kappa_2}{\beta}$, and
\begin{align*}
\bsM_{1}&=(\alpha-\beta)\bsL-\frac{1}{2}(2+3L_f^2)\bsK,\\
\bsM_{2}&=\beta^2\bsL+(2\alpha^2+\beta^2)\bsL^2+4L_f^2\bsK,\\
\kappa_3&= \frac{1}{\rho_2(L)} + \kappa_1 + 1,\\
b_{1}&=\frac{1}{2}(2\beta-\kappa_3)\eta
-\frac{5}{2}\beta^2\eta^2,\\
b_{2}&=\frac{1}{4}-b_{3}\eta,\\
b_{3}&=L_f+\frac{1}{\beta^2\eta}\kappa_3L_f^2+\frac{2}{\beta^2}(\kappa_3+1)L_f^2.		
\end{align*}
Consider $p\geq 1$, $\alpha=\kappa_1\beta$, $\kappa_1>1$, $\beta$ is large enough, and $\eta=\frac{\kappa_2}{\beta}$, we have
\begin{align}
\eta\bsM_{1}&\ge [(\kappa_1 - 1)\rho_2(L) - 1]\kappa_2\bsK.\label{fo_pb:m1-rand-pd}\\
\eta^2\bsM_{2}&\le [\rho(L)+(2\kappa_1^2+1)\rho(L^2)+1]\kappa_2^2\bsK.\label{fo_pb:m2-rand-pd}\\
b^0_{2}&\ge\frac{1}{2}(\kappa_2 - 5\kappa_2^2).\label{fo_pb:vkLya-b1}
\end{align}
From \eqref{fo_pb:vkLya}--\eqref{fo_pb:vkLya-b1}, let $\kappa_4 = [(\kappa_1 - 1)\rho_2(L) - 1]\kappa_2- [\rho(L)+(2\kappa_1^2+1)\rho(L^2)+1]\kappa_2^2 $ we know that \eqref{fo_pb:sgproof-vkLya2T} holds.
Similar to the way to get \eqref{fo_pb:sgproof-vkLya2T}, we have \eqref{fo_pb:sgproof-vkLya2T_W4}.
\end{proof}

\onecolumn
\section{Networks Architectures and experiment parameters}
\label{app:exp}
\begin{figure*}[!ht]
	\centering
	\begin{tikzpicture}[shorten >=1pt]
		\tikzstyle{unit}=[draw,shape=circle,minimum size=1.15cm]
		\tikzstyle{hidden}=[draw,shape=circle,minimum size=1.15cm]
 
		\node[unit](x0) at (0,3.5){$x_0$};
		\node[unit](x1) at (0,2){$x_1$};
		\node at (0,1){\vdots};
		\node[unit](xd) at (0,0){$x_{400}$};
 
		\node[hidden](h10) at (3,4){$z_0^{(1)}$};
		\node[hidden](h11) at (3,2.5){$z_1^{(1)}$};
		\node at (3,1.5){\vdots};
		\node[hidden](h1m) at (3,-0.5){$z_{50}^{(1)}$};
 
		\node(h22) at (5,0){};
		\node(h21) at (5,2){};
		\node(h20) at (5,4){};
		
		\node(d3) at (6,0){$\ldots$};
		\node(d2) at (6,2){$\ldots$};
		\node(d1) at (6,4){$\ldots$};
 
		\node(hL12) at (7,0){};
		\node(hL11) at (7,2){};
		\node(hL10) at (7,4){};
		
		\node[hidden](hL0) at (9,4){$z_0^{(2)}$};
		\node[hidden](hL1) at (9,2.5){$z_1^{(2)}$};
		\node at (9,1.5){\vdots};
		\node[hidden](hLm) at (9,-0.5){$z_{10}^{(2)}$};
 
		\node[unit](y1) at (12,3.5){$y_1$};
		\node[unit](y2) at (12,2){$y_2$};
		\node at (12,1){\vdots};	
		\node[unit](yc) at (12,0){$y_{10}$};
 
		\draw[->] (x0) -- (h11);
		\draw[->] (x0) -- (h1m);
 
		\draw[->] (x1) -- (h11);
		\draw[->] (x1) -- (h1m);
 
		\draw[->] (xd) -- (h11);
		\draw[->] (xd) -- (h1m);
 
		\draw[->] (hL0) -- (y1);
		\draw[->] (hL0) -- (yc);
		\draw[->] (hL0) -- (y2);
 
		\draw[->] (hL1) -- (y1);
		\draw[->] (hL1) -- (yc);
		\draw[->] (hL1) -- (y2);
 
		\draw[->] (hLm) -- (y1);
		\draw[->] (hLm) -- (y2);
		\draw[->] (hLm) -- (yc);
 
		\draw[->,path fading=east] (h10) -- (h21);
		\draw[->,path fading=east] (h10) -- (h22);
		
		\draw[->,path fading=east] (h11) -- (h21);
		\draw[->,path fading=east] (h11) -- (h22);
		
		\draw[->,path fading=east] (h1m) -- (h21);
		\draw[->,path fading=east] (h1m) -- (h22);
		
		\draw[->,path fading=west] (hL10) -- (hL1);
		\draw[->,path fading=west] (hL11) -- (hL1);
		\draw[->,path fading=west] (hL12) -- (hL1);
		
		\draw[->,path fading=west] (hL10) -- (hLm);
		\draw[->,path fading=west] (hL11) -- (hLm);
		\draw[->,path fading=west] (hL12) -- (hLm);
		
		\draw [decorate,decoration={brace,amplitude=10pt},xshift=-4pt,yshift=0pt] (-0.5,4) -- (0.75,4) node [black,midway,yshift=+0.6cm]{input layer};
		\draw [decorate,decoration={brace,amplitude=10pt},xshift=-4pt,yshift=0pt] (2.5,4.5) -- (3.75,4.5) node [black,midway,yshift=+0.6cm]{$1^{\text{st}}$ hidden layer};
		\draw [decorate,decoration={brace,amplitude=10pt},xshift=-4pt,yshift=0pt] (8.5,4.5) -- (9.75,4.5) node [black,midway,yshift=+0.6cm]{$2^{\text{nd}}$ hidden layer};
		\draw [decorate,decoration={brace,amplitude=10pt},xshift=-4pt,yshift=0pt] (11.5,4) -- (12.75,4) node [black,midway,yshift=+0.6cm]{output layer};
	\end{tikzpicture}
	\caption{Neural Network Architecture.}
	\label{fig:nn}
\end{figure*}
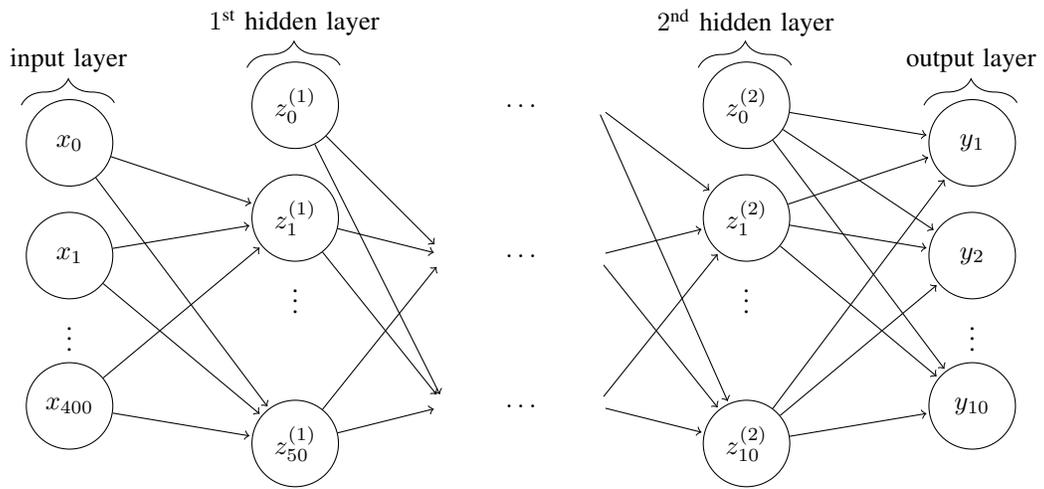

\begin{figure*}[!ht]
\centering
  \includegraphics[width=1\textwidth]{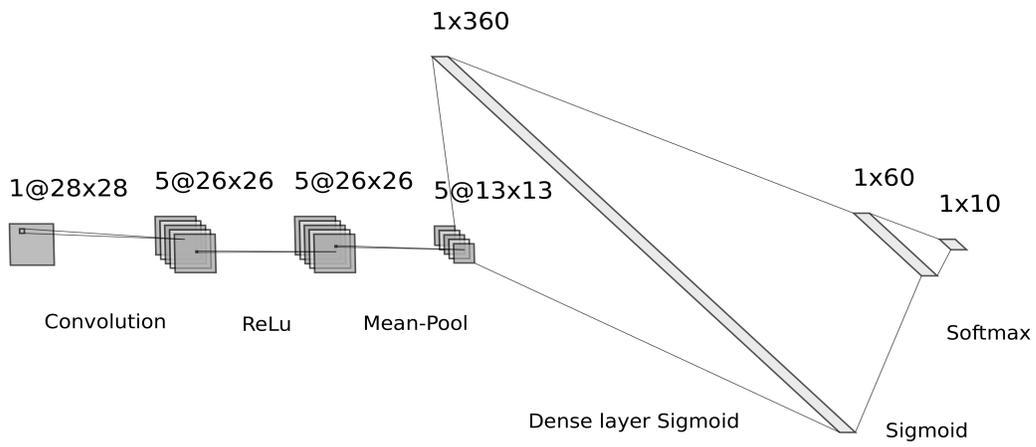}
  \caption{CNN Architecture.}
  \label{fig:cnn}
\end{figure*}

\begin{table*}[!ht]
\caption{{Parameters for each algorithm in NN experiment.} }
\label{tab:nn-par}
\vskip 0.15in
\begin{center}
\begin{scriptsize}
\begin{tabular}{M{2.1cm}|M{1.2cm}|M{2.0cm}|M{2.5cm} |M{2.5cm}}
\hline
Algorithm & $\eta$ & $\alpha$ & $\beta$ & $\gamma$ \\
\hline
DSGPA-T-PB & $0.08/{k^{10^{-5}}}$ & $4k^{10^{-5}}$ & $3k^{10^{-5}}$ & 0.7\\
\hline
DSGPA-F-PB &0.03 & 5 & 20 & 0.7\\
\hline
DSGPA-T & $0.08/{k^{10^{-5}}}$ & $4k^{10^{-5}}$ & $3k^{10^{-5}}$ & \ding{55}\\
\hline
DSGPA-F &0.03 & 5 & 20 & \ding{55}\\
\hline
DM-SGD \cite{Yu2019on} & 0.1 & \ding{55} &0.8 & \ding{55}\\
\hline
D-SGD-1 \cite{jiang2017collaborative,lian2017can}& 0.1 &\ding{55} &\ding{55}& \ding{55}\\
\hline
D-SGD-2 \cite{george2019distributed} &\ding{55} & $0.1/(10^{-5}k + 1)$&$0.2/(10^{-5}k + 1)^{0.3}$& \ding{55}\\
\hline
$D^{2}$ \cite{Tang2018Decentralized}& 0.01 &\ding{55} &\ding{55}& \ding{55}\\
\hline
D-SGT-1 \cite{lu2019gnsd,xin2019distributed} & 0.01& \ding{55}&\ding{55}& \ding{55}\\
\hline
D-SGT-2 \cite{zhang2019decentralized,pu2018distributed} & 0.01 & \ding{55}&\ding{55}& \ding{55}\\
\hline
C-SGD & 0.1 &\ding{55} &\ding{55}& \ding{55}\\
\hline
\end{tabular}
\end{scriptsize}
\end{center}
\vskip -0.1in
\end{table*}
\begin{table*}[!ht]
\caption{{Parameters for each algorithm in CNN experiment.} }
\label{tab:cnn-par}
\vskip 0.15in
\begin{center}
\begin{scriptsize}
\begin{tabular}{M{2.1cm}|M{1.2cm}|M{2.0cm}|M{2.5cm}|M{2.5cm}}
\hline
Algorithm & $\eta$ & $\alpha$ & $\beta$ & $\gamma$\\
\hline
DSGPA-F-PB &0.5 & 0.5 & 0.1 & 0.5\\
\hline
DSGPA-T & $0.5/{k^{10^{-5}}}$ & $0.5k^{10^{-5}}$ & $0.1k^{10^{-5}}$& \ding{55}\\
\hline
DSGPA-F &0.5 & 0.5 & 0.1& \ding{55}\\
\hline
DM-SGD \cite{Yu2019on} & 0.1 & \ding{55} &0.8& \ding{55}\\
\hline
D-SGD \cite{jiang2017collaborative,lian2017can}& 0.1 &\ding{55} &\ding{55}& \ding{55}\\
\hline
C-SGD & 0.1 &\ding{55} &\ding{55}& \ding{55}\\
\hline
\end{tabular}
\end{scriptsize}
\end{center}
\vskip -0.1in
\end{table*}
\end{appendices}
\end{document}